\documentclass[11pt]{article}

\usepackage{amssymb,amsmath,amsthm}
\usepackage[english]{babel}
\usepackage{t1enc}
\usepackage[latin2]{inputenc}
\usepackage{epsfig}
\usepackage{subfigure}

\newtheorem{thm}{Theorem}

\newtheorem{defi}[thm]{Definition}
\newtheorem{lemma}[thm]{Lemma}
\newtheorem{prop}[thm]{Proposition}
\newtheorem{claim}[thm]{Claim}

\usepackage{indentfirst}
\begin{document}

\def\marrow{{\boldmath {\marginpar[\hfill$\rightarrow \rightarrow$]{$\leftarrow \leftarrow$}}}}
\def\dom#1{{\sc D\"om\"ot\"or: }{\marrow\sf #1}}
\def\balazs#1{{\sc Bal\'azs: }{\marrow\sf #1}}
\def\janos#1{{\sc Ja\'nos: }{\marrow\sf #1}}

\def\eps{$\epsilon$ }

\title{Drawing planar graphs of bounded degree\\ with few slopes}

\author{Bal\'azs Keszegh\thanks{EPFL, Lausanne and R\'enyi Institute, Budapest} \and
J\'anos Pach\thanks{EPFL, Lausanne and R\'enyi Institute, Budapest.
Supported by Grants from NSF, SNF, NSA, PSC-CUNY, and OTKA.}\and
D\"om\"ot\"or P\'alv\"olgyi \thanks{EPFL, Lausanne and E\"otv\"os University, Budapest}}


\maketitle

\begin{abstract}
We settle a problem of Dujmovi\'c, Eppstein, Suderman, and Wood by showing that there exists a function $f$ with the property that every planar graph $G$ with maximum degree $d$ admits a drawing with noncrossing straight-line edges, using at most $f(d)$ different slopes. If we allow the edges to be represented by polygonal paths with {\em one} bend, then $2d$ slopes suffice. Allowing {\em two} bends per edge, every planar graph with maximum degree $d\ge 3$ can be drawn using segments of at most $\lceil d/2\rceil$ different slopes. There is only one exception: the graph formed by the edges of an octahedron is 4-regular, yet it requires 3 slopes. These bounds cannot be improved.
\end{abstract}

\section{Introduction}
A planar layout of a graph $G$ is called a {\em drawing} if the vertices of $G$ are represented  by distinct points in the plane and every edge is represented by a continuous arc connecting the corresponding pair of points and not passing through any other point representing a vertex \cite{DETT99}. If it leads to no confusion, in notation and terminology we make no distinction between a vertex and the corresponding point and between an edge and the corresponding arc. If the edges are represented by line segments, the drawing is called a {\em straight-line drawing}. The {\em slope} of an edge in a straight-line drawing is the slope of the corresponding segment.

In this paper, we will be concerned with drawings of planar graphs. Unless it is stated otherwise, all {\em drawings} will be {\em noncrossing}, that is, no two arcs that represent different edges have an interior point in common.

Every planar graph admits a straight-line drawing \cite{F48}. From the practical and aesthetical point of view, it makes sense to minimize the number of slopes we use \cite{WC94}.
The {\em planar slope number} of a planar graph $G$ is the smallest number $s$
with the property that $G$ has a straight-line drawing with edges
of at most $s$ distinct slopes. If $G$ has a vertex of degree $d$, then its planar slope number is at least $\lceil d/2\rceil$, because in a straight-line drawing no two edges are allowed to overlap.

Dujmo\-vi\'c, Eppstein, Suderman, and Wood \cite{DESW07} raised the question whether there exists a function $f$ with the property that the planar slope number of every planar graph with maximum degree $d$ can be bounded from above by $f(d)$. Jelinek et al.~\cite{JJ10} have shown that the answer is yes for {\em outerplanar} graphs, that is, for planar graphs that can be drawn so that all of their vertices lie on the outer face. In Section 2, we answer this question in full generality. We prove the following.

\begin{thm}\label{one} Every planar graph with maximum degree $d$ admits a straight-line drawing, using segments of $O(d^2(3+2\sqrt 3)^{12d})\le K^d$ distinct slopes.
\end{thm}

The proof is based on a paper of Malitz and Papakostas \cite{MP94}, who used Koebe's theorem \cite{K36} on disk representations of planar graphs to prove the existence of drawings with relatively large angular resolution. As the proof of these theorems, our argument is nonconstructive; it only yields a nondeterministic algorithm with running time $O(dn)$. However, if one combines our result with a polynomial time algorithm that computes the $\epsilon$-approximation of the disk representation (see e.g. Mohar \cite{M93}), then one can obtain a deterministic algorithm running in time exponential in $d$ but polynomial in $n$.

For $d=3$, much stronger results are known than the one given by our theorem. Dujmovi\'c at al. \cite{DESW07} showed that every planar graph with maximum degree 3 admits a straight-line drawing using at most 3 different slopes, except for at most 3 edges of the outer face, which may require 3 additional slopes. This complements Ungar's old theorem \cite{U53}, according to which 3-regular, 4-edge-connected planar graphs require only 2 slopes and 4 extra edges.

The exponential upper bound in Theorem~\ref{one} is probably far from being optimal. However, we were unable to give any superlinear lower bound for the largest planar slope number of a planar graph with maximum degree $d$. The best constructions we are aware of are presented in Section 5.

It is perhaps somewhat surprising that if we do not restrict our attention to planar graphs, then no result similar to Theorem~\ref{one} holds.
For every $d\ge 5$, Bar\'at, Matou\v sek, and Wood \cite{BMW06} and,
independently, Pach and P\'alv\"olgyi \cite{PP06} constructed graphs with maximum degree $d$ with the property that no matter how we draw them in the plane with (possibly crossing) straight-line edges, we must use an arbitrarily large number of slopes. (See also \cite{DSW07}.) The case $d\le 3$ is different: Keszegh et al. \cite{KPPT08} proved that every graph with maximum degree 3 can be drawn with 5 slopes. Moreover, Mukkamala and Szegedy \cite{MSz07} showed that 4 slopes suffice if the graph is connected.
The case $d=4$ remains open.

Returning to planar graphs, we show that significantly fewer slopes are sufficient if we are allowed to represent the edges by short noncrossing poly\-gonal paths. If such a path consists of $k+1$ segments, we say that the edge is drawn by $k$ {\em bends}. In Section 3, we show if we allow one bend per edge, then every planar graph can be drawn using segments with $O(d)$ slopes.

\begin{thm}\label{onebend}
Every planar graph $G$ with maximum degree $d$ can be drawn with at most $1$ bend per edge, using at most $2d$ slopes.
\end{thm}

Allowing {\em two} bends per edge yields an optimal result: almost all planar graphs with maximum degree $d$ can be drawn with $\left\lceil d/2 \right\rceil$ slopes. In Section 4, we establish

\begin{thm}\label{twobendsimproved}\label{twobends}
Every planar graph $G$ with maximum degree $d\ge 3$ can be drawn with at most 2 bends per edge, using segments of at most $\lceil d/2\rceil$ distinct slopes. The only exception is the graph formed by the edges of an octahedron, which is 4-regular, but requires 3 slopes. These bounds are best possible.
\end{thm}

It follows from the proof of Theorem~\ref{twobends} that in the cyclic order of directions, the slopes of the edges incident to any given vertex form a contiguous interval. Moreover, the $\lceil d/2\rceil$ directions we use can be chosen to be equally spaced in $[0,2\pi)$. We were unable to guarantee such a nice property in Theorem~\ref{onebend}: even for a fixed $d$, as the number of vertices increases, the smallest difference between the $2d-2$ slopes we used tends to zero. We suspect that this property is only an unpleasant artifact of our proof technique.

\section{Straight-line drawings--Proof of Theorem~\ref{one}}

Note that it is sufficient to prove the theorem for triangulated planar graphs,  because any planar graph can be triangulated by adding vertices and edges so that the degree of each vertex increases only by a factor of at most three \cite{PT06}, so at the end we will lose this factor.

We need the following result from \cite{MP94}, which is not displayed as a theorem there, but is stated right above Theorem 2.2.

\begin{lemma}\label{mp} {\rm (Malitz-Papakostas)} The vertices of any triangulated planar graph $G$ with maximum degree $d$ can be represented by nonoverlapping disks in the plane so that two disks are tangent to each other if and only if the corresponding vertices are adjacent, and the ratio of the radii of any two disks that are tangent to each other is at least $\alpha^{d-2},$ where $\alpha=\frac 1{3+2\sqrt 3}\approx 0.15$.
\end{lemma}

Lemma~\ref{mp} can be established by taking any representation of the vertices of $G$ by tangent disks, as guaranteed by Koebe's theorem, and applying a conformal mapping to the plane that takes the disks corresponding to the three vertices of the outer face to disks of the same radii. The lemma now follows by the observation that any internal disk is surrounded by a ring of at most $d$ mutually touching disks, and the radius of none of them can be much smaller than that of the central disk.

The idea of the proof of Theorem \ref{one} is as follows. Let $G$ be a triangulated planar graph with maximum degree $d$, and denote its vertices by $v_1, v_2, \ldots$. Consider a disk representation of $G$ meeting the requirements of Lemma~\ref{mp}. Let $D_i$ denote the disk that represents $v_i$, and let $O_i$ be the center of $D_i$. By properly scaling the picture if necessary, we can assume without loss of generality that the radius of the smallest disk $D_i$ is sufficiently large. Place an integer grid on the plane, and replace each center $O_i$ by the nearest grid point. Connecting the corresponding pairs of grid points by segments, we obtain a straight-line drawing of $G$. The advantage of using a grid is that in this way we have control of the slopes of the edges. The trouble is that the size of the grid, and thus the number of slopes used, is very large. Therefore, in the neighborhood of each disk $D_i$, we use a portion of a grid whose side length is proportional to the radius of the disk. These grids will nicely fit together, and each edge will connect two nearby points belonging to grids of comparable sizes. Hence, the number of slopes used will be bounded. See Figure \ref{fig:discs}.

\begin{figure}[ht]
\centering
		\includegraphics[scale=0.65]{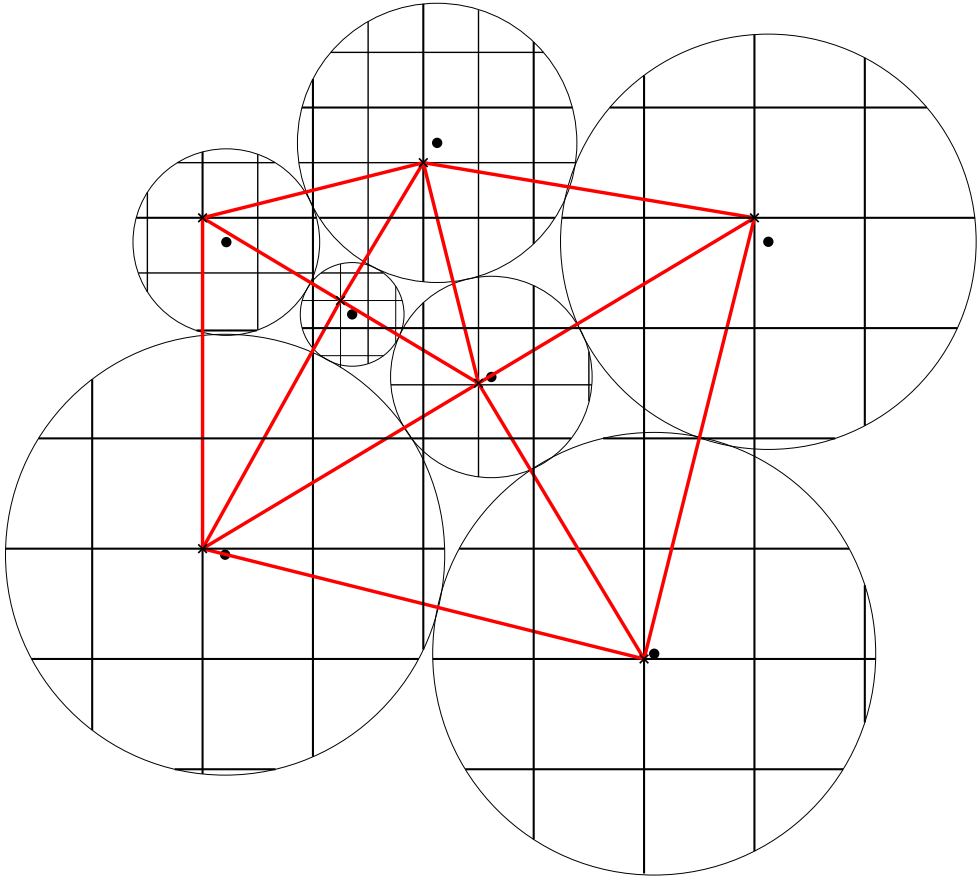}
               \hspace{5mm}
		 \caption{Straight-line graph from disk representation}
		\label{fig:discs}
\end{figure}

Now we work out the details. Let $r_i$ denote the radius of $D_i\; (i=1,2\ldots)$, and suppose without loss of generality that $r^*$, the radius of the smallest disk is
$$r^*=min_i r_i=\sqrt{2}/\alpha^{d-2}>1,$$
where $\alpha$ denotes the same constant as in Lemma~\ref{mp}.

Let $s_i=\lfloor\log_d (r_i/r^*)\rfloor\ge 0$, and represent each vertex $v_i$ by the integer point nearest to $O_i$ such that both of its coordinates are divisible by $d^{s_i}$. (Taking a coordinate system in general position, we can make sure that this point is unique.) For simplicity, the point representing $v_i$ will also be denoted by $v_i$. Obviously, we have that the distance between $O_i$ and $v_i$ satisfies
$$\overline{O_iv_i}< \frac{d^{s_i}}{\sqrt{2}}.$$

Since the centers $O_i$ of the disks induce a (crossing-free) straight-line drawing of $G$, in order to prove that moving the vertices to $v_i$ does not create a crossing, it is sufficient to verify the following statement.

\begin{lemma} For any three mutually adjacent vertices, $v_i, v_j, v_k$ in $G$, the orientation of the triangles $O_iO_jO_k$ and $v_iv_jv_k$ are the same.
\end{lemma}

\begin{proof} By Lemma~\ref{mp}, the ratio between the radii of any two adjacent disks is at least $\alpha^{d-2}$. Suppose without loss of generality that $r_i\ge r_j\ge r_k\ge \alpha^{d-2}r_i$.
For the orientation to change, at least one of $\overline{O_iv_i}$, $\overline{O_jv_j}$, or $\overline{O_kv_k}$ must be at least half of the smallest altitude of the triangle $O_iO_jO_k$, which is at least $\frac{r_k}{2}$.

On the other hand, as we have seen before, each of these numbers is smaller than
$$\frac{d^{s_i}}{\sqrt{2}}\le\frac{r_i/r^*}{\sqrt{2}}
=\frac{{\alpha}^{d-2}r_i}{2}\le \frac{r_k}{2}$$
which completes the proof.
\end{proof}

Now we are ready to complete the proof of Theorem~\ref{one}. Take an edge $v_iv_j$ of $G$, with $r_i\ge r_j\ge \alpha^{d-2}r_i$.
The length of this edge can be bounded from above by
$$\overline{v_iv_j}\le  \overline{O_iO_j}+\overline{O_iv_i}+\overline{O_jv_j}
\le r_i+r_j+\frac{d^{s_i}}{\sqrt{2}}+\frac{d^{s_j}}{\sqrt{2}}
\le 2r_i+{\sqrt{2}}{d^{s_i}}
\le 2r_i+{\sqrt{2}}{r_i/r^*}$$
$$ \le r_i/r^*(2r^*+\sqrt{2})
\le \frac{r_j/r^*}{\alpha^{d-2}}(2r^*+\sqrt{2})
< \frac{d^{s_j+1}}{\alpha^{d-2}}(\frac{2\sqrt{2}}{\alpha^{d-2}}+\sqrt{2}).$$

According to our construction, the coordinates of $v_j$ are integers divisible by $d^{s_j}$, and the coordinates of $v_i$ are integers divisible by $d^{s_i}\ge d^{s_j}$, thus also by $d^{s_j}$.

Thus, shrinking the edge $v_iv_j$ by a factor of $d^{s_j}$, we obtain a segment whose endpoints are integer points at a distance at most
$\frac{d}{\alpha^{d-2}}(\frac{2\sqrt{2}}{\alpha^{d-2}}+\sqrt{2}).$
Denoting this number by $R(d)$, we obtain that the number of possible slopes for $v_iv_j$, and hence for any other edge in the embedding, cannot exceed the number of integer points in a disk of radius $R(d)$ around the origin. Thus, the planar slope number of any triangulated planar graph of maximum degree $d$ is at most roughly $R^2(d)\pi=O(d^2/\alpha^{4d})$, which completes the proof. \hfill $\square$

\medskip

Our proof is based on the result of Malitz and Papakostas that does not have an algorithmic version. However, with some reverse engineering, we can obtain a nondeterministic algorithm for drawing a triangulated planar graph of bounded degree with a bounded number of slopes. Because of the enormous constants in our expressions, this algorithm is only of theoretical interest. Here is a brief sketch.
\medskip

{\em Nondeterministic algorithm.} First, we guess the three vertices of the outer face and their coordinates in the grid scaled according to their radii. Then embed the remaining vertices one by one. For each vertex, we guess the radius of the corresponding disk as well as its coordinates in the proportionally scaled grid. This algorithm runs in nondeterministic $O(dn)$ time.

\section{One bend per edge--Proof of Theorem~\ref{onebend}}

In this section, we represent edges by noncrossing polygonal paths, each consisting of at most two segments. Our goal is to establish Theorem~\ref{onebend}, which states that the total number of directions assumed by these segments grows at most linearly in $d$.

The proof of Theorem~\ref{onebend} is based on a result of Fraysseix et al. \cite{FMR94}, according to which every planar graph can be represented as a contact graph of $T$-shapes. A {\em $T$-shape} consists of a vertical and a horizontal segment such that the upper endpoint of the vertical segment lies in the interior of the horizontal segment. The vertical and horizontal segments of $T$ are called its {\em leg} and {\em hat}, while their point of intersection is the {\em center} of the $T$-shape. The two endpoints of the hat and the bottom endpoint of the leg are called {\em ends} of the $T$-shape.

Two $T$-shapes are {\em noncrossing} if the interiors of their segments are disjoint. Two $T$-shapes are {\em tangent} to each other if they are noncrossing but they have a point in common.

\begin{lemma}\label{Tshapes} {\rm (Fraysseix et al.)} The vertices of any planar graph with $n$ vertices can be represented by noncrossing $T$-shapes such that
\begin{enumerate}
\item two $T$-shapes are tangent to each other if and only if the corresponding vertices are adjacent;
\item the centers and the ends of the $T$-shapes belong to an $n\times n$ grid.
\end{enumerate}
Moreover, such a representation can be computed in linear time.
\end{lemma}

The proof of the lemma is based on the canonical ordering of the vertices of a planar graph, introduced in \cite{FPP89}.

\begin{proof}[Proof of Theorem \ref{onebend}]
Consider a representation of $G$ by $T$-shapes satisfying the conditions in the lemma. See Figure \ref{fig:tshapesleft}. For any $v\in V(G)$, let $T_v$ denote the corresponding $T$-shape. We define a drawing of $G$, in which the vertex $v$ is mapped to the center of $T_v$. To simplify the presentation, the center of $T_v$ is also denoted by $v$. For any $uv\in E(G)$, let $p_{uv}$ denote the point of tangency of $T_u$ and $T_v$. The polygonal path $up_{uv}v$ consists of a horizontal and a vertical segment, and these paths together almost form a drawing of $G$ with one bend per edge, using segments of two different slopes. The only problem is that these paths partially overlap in the neighborhoods of their endpoints. Therefore, we modify them by replacing their horizontal and vertical pieces by almost horizontal and almost vertical ones, as follows.

For any $1\le i\le d$, let $\alpha_i$ denote the slope of the (almost horizontal) line connecting the origin $(0,0)$ to the point $(2in,-1)$. Analogously, let $\beta_i$ denote the slope of the (almost vertical) line passing through $(0,0)$ and $(1,2in)$.

Fix a $T$-shape $T_v$ in the representation of $G$. It is tangent to at most $d$ other $T$-shapes. Starting at its center $v$, let us pass around $T_v$ in the counterclockwise direction, so that we first visit the upper left side of its hat, then its lower left side, then the left side and right side of its leg, etc. We number the points of tangencies along $T_v$ in this order. (Note that there are no points of tangencies on the lower side of the hat.)

Suppose now that the hat of a $T$-shape $T_u$ is tangent to the leg of $T_v$, and let $p_{uv}$ be their point of tangency. Assume that $p_{uv}$ was the number $i$ point of tangency along $T_u$ and the number $j$ point of tangency along $T_v$. Let $p'_{uv}$ denote the unique point of intersection of the (almost horizontal) line through $u$ with slope $\alpha_i$ and the (almost vertical) line through $v$ with slope $\beta_j$. In our drawing of $G$, the edge $uv$ will be represented by the polygonal path $up'_{uv}v$. See Figure \ref{fig:tshapesright} for the resulting drawing and Figure \ref{fig:tshapesmiddle} for a version distorted for the human eye to show the underlying structure.

\begin{figure}[ht]
\centering
  \subfigure[]
{		\includegraphics[scale=0.4]{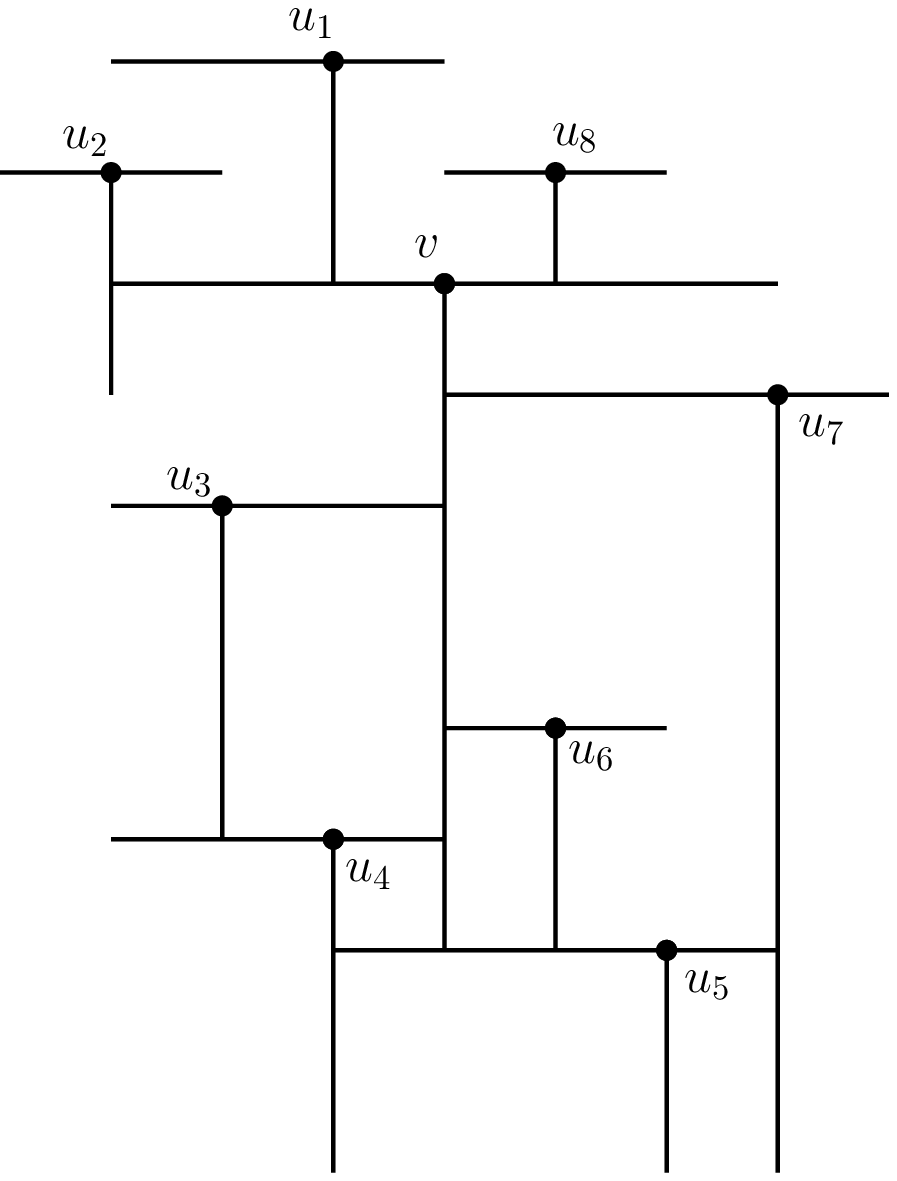}
		\hspace{1cm}
				\label{fig:tshapesleft}
	}
	\subfigure[]
{
		\includegraphics[scale=0.5]{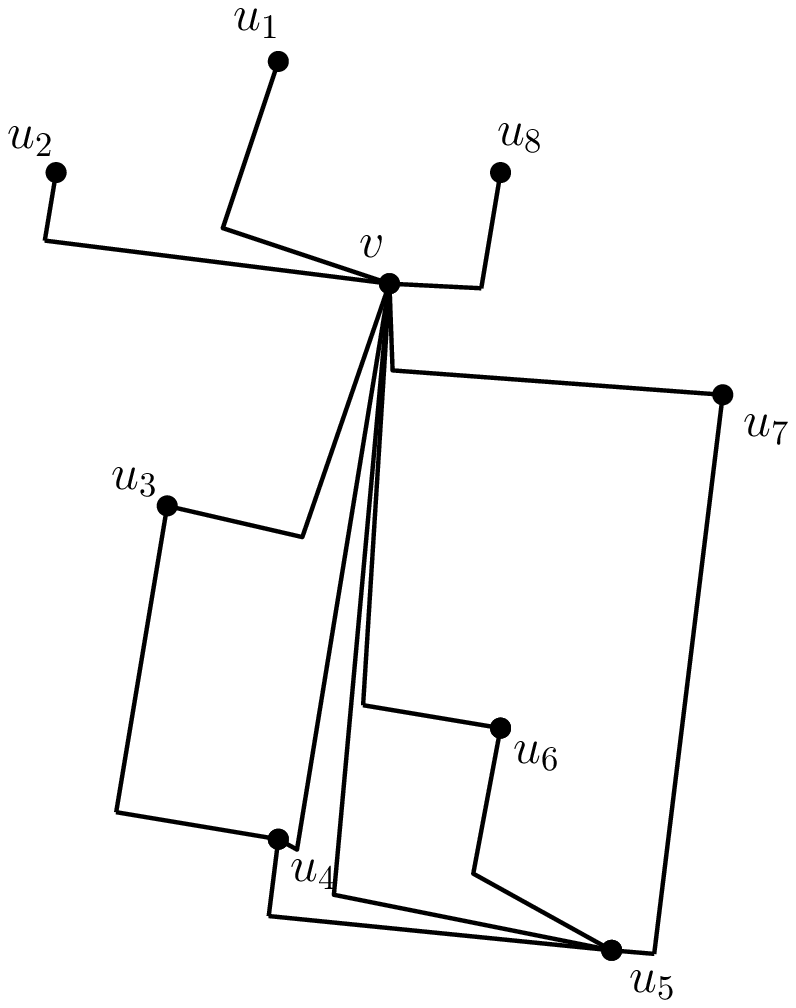}
		\hspace{1cm}
				\label{fig:tshapesmiddle}
		}	
  \subfigure[]
{\includegraphics[scale=0.5]{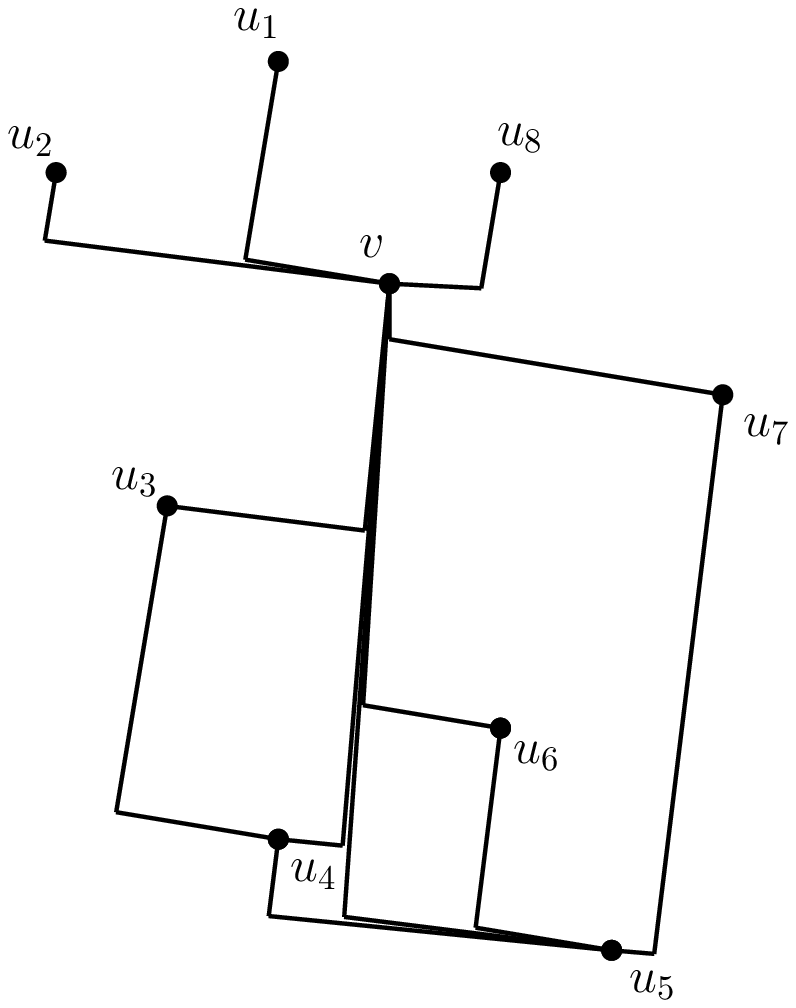}
               \hspace{5mm}
		 		\label{fig:tshapesright}
		 }
		 \caption{Representation with $T$-shapes and the drawing with one bend per edge}
		\label{fig:tshapes}
\end{figure}

Since the segments we used are almost horizontal or vertical, the modified edges $up'_{uv}v$ are very close (within distance $1/2$) of the original polygonal paths $up_{uv}v$. Thus, no two nonadjacent edges can cross each other. On the other hand, the order in which we picked the slopes around each $v$ guarantees that no two edges incident to $v$ will cross or overlap. This completes the proof.
\end{proof}

\section{Two bends per edge--Proof of Theorem~\ref{twobendsimproved}}

In this section, we draw the edges of a planar graph by polygonal paths with at most {\em two} bends. Our aim is to establish Theorem~\ref{twobendsimproved}.

Note that the statement is trivially true for $d=1$ and is false for $d=2$. It is sufficient to prove Theorem~\ref{twobendsimproved} for even values of $d$. For $d=4$, the assertion was first proved by Liu et al. \cite{LMS91} and later, independently, by Biedl and Kant \cite{BK98} (also that the only exception is the octahedral graph). The latter approach is based on the notion of $st$-ordering of biconnected ($2$-connected) graphs from Lempel et al. \cite{LEC67}. We will show that this method generalizes to higher values of $d\ge 5$. As it is sufficient to prove the statement for even values of $d$, from now on we suppose that $d\ge 6$ even. 
We will argue that it is enough to consider {\em biconnected} graphs. Then we review some crucial claims from \cite{BK98} that will enable us to complete the proof. We start with some notation.

Take $d\ge 5$ lines that can be obtained from a vertical line by clockwise rotation by $0, \pi/d, 2\pi/d,$ $\ldots, (d-1)\pi/d$ degrees. Their slopes are called the $d$ {\em regular slopes}. We will use these slopes to draw $G$. Since these slopes depend only on $d$ and not on $G$, it is enough to prove the theorem for connected graphs. If a graph is not connected, its components can be drawn separately. 

In this section we always use the term ``slope'' to mean a regular slope. The {\em  directed slope} of a directed line or segment is defined as the angle (mod $2\pi$) of a clockwise rotation that takes it to a position parallel to the upward directed $y$-axis. Thus, if the directed slopes of two segments differ by $\pi$, then they have the same slope. We say that the slopes of the segments incident to a point $p$ form a {\em contiguous interval} if the set $\vec S\subset \{0, \pi/d, 2\pi/d,\ldots, (2d-1)\pi/d\}$ of directed slopes of the segments directed away from $p$, has the property that for all but at most one $\alpha\in \vec S$, we have that $\alpha +\pi/d \mod 2\pi \in \vec S$ (see Figure \ref{fig:twobendsboth}).

Finally, we say that $G$ admits a {\em good drawing} if $G$ has a planar drawing such that every edge has at most $2$ bends, every segment of every edge has one of the $\lceil d/2 \rceil$ regular slopes, and the slopes of the segments incident to any vertex form a contiguous interval.
If $t$ is a vertex whose degree is at least two but less than $d$, then we can define the two {\em extremal segments} at $t$ as the segments corresponding to the slopes at the two ends of the contiguous interval formed by the slopes of all the segments incident to $t$.
Also define the {\em $t$-wedge} as the infinite cone bounded by the extension of the two extremal segments, which contains all segments incident to $t$ and none of the ``missing'' segments. See Figure \ref{fig:tcone}. For a degree one vertex $t$ we define the {\em $t$-wedge} as the infinite cone bounded by the extension of the rotations of the segment incident to $t$ around $t$ by $\pm\pi/2d$.

\begin{figure}[ht]
\centering
  \subfigure[]
{		\includegraphics[scale=0.4]{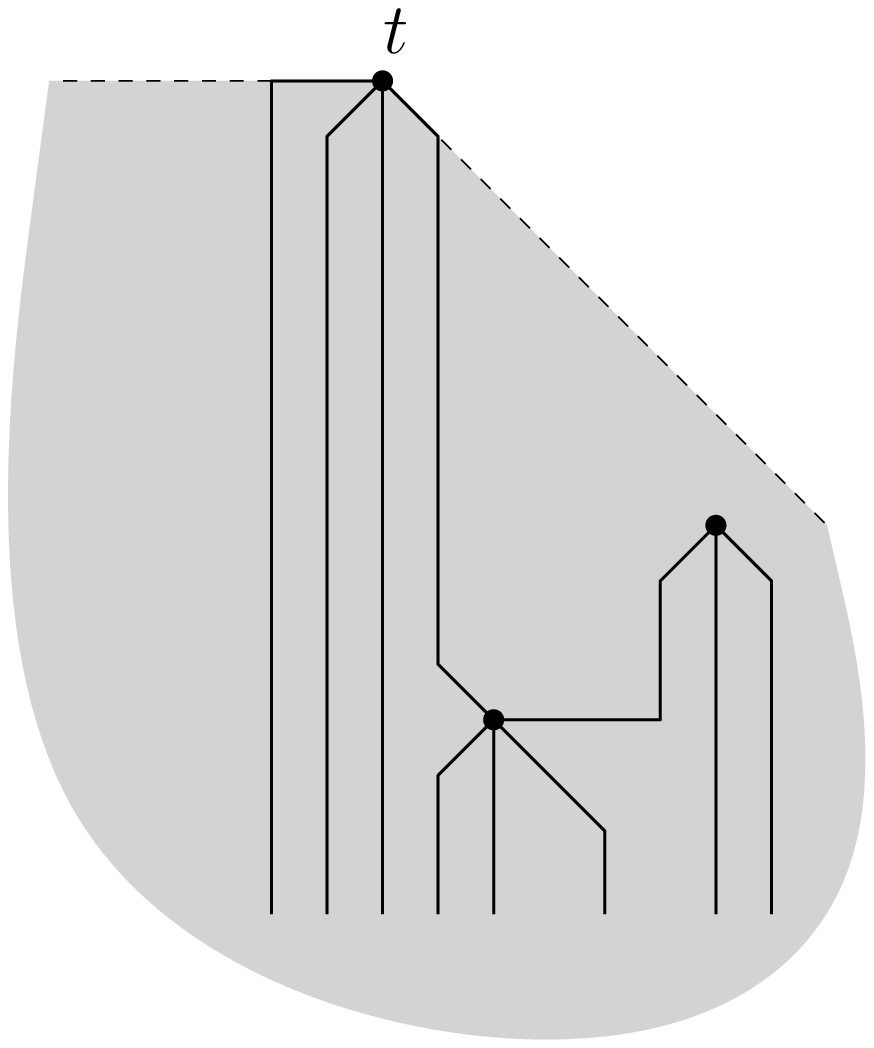}
		\hspace{2cm}
				\label{fig:tcone2}
	}
	\subfigure[]
{
		\includegraphics[scale=0.4]{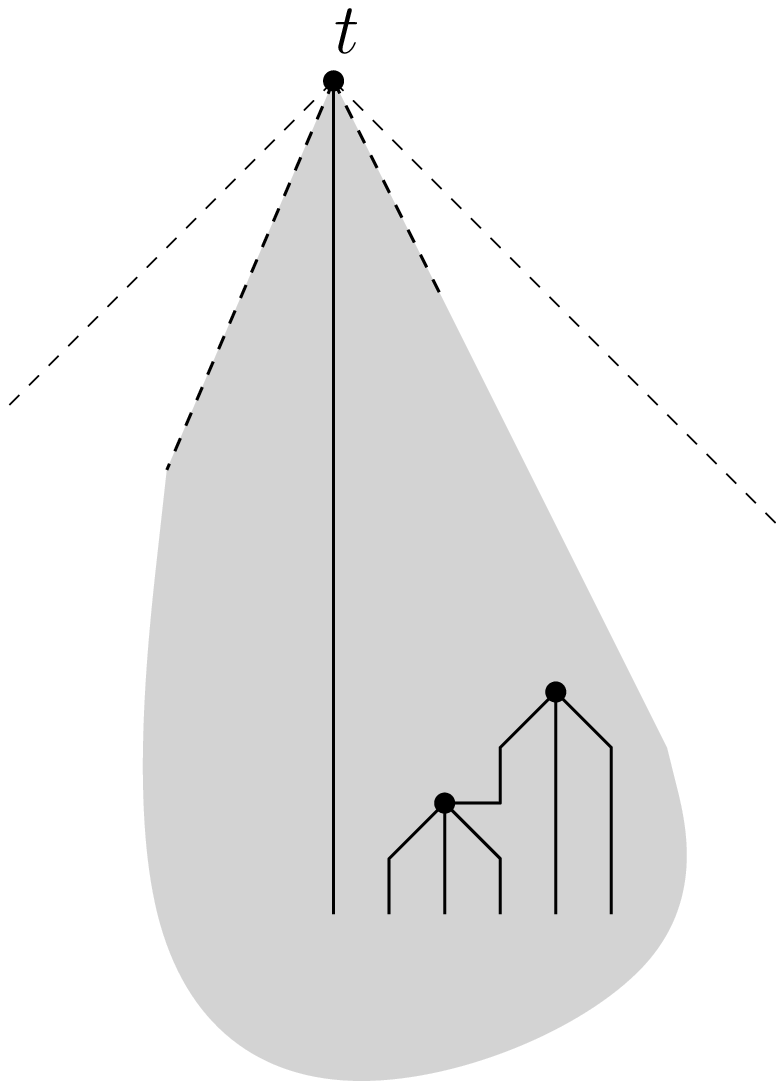}
						\label{fig:tcone1}
		}	
		
		 \caption{The $t$-wedge}
		 		\label{fig:tcone}
\end{figure}

To prove Theorem~\ref{twobendsimproved}, we show by induction that every connected planar graph with maximum degree $d\ge 6$ with an arbitrary $t$ vertex whose degree is strictly less than $d$ admits a good drawing that is contained in the $t$-wedge. Note that such a vertex always exist because of Euler's polyhedral formula, thus Theorem \ref{twobends} is indeed a direct consequence of this statement. First we show how the induction step goes for graphs that have a cut vertex, then (after a lot of definitions) we prove the statement also for biconnected graphs (without the induction hypothesis).

\begin{lemma}\label{cutclaim} Let $G$ be a connected planar graph of maximum degree $d$, let $t\in V(G)$ be a vertex whose degree is strictly smaller than $d$, and let $v\in V(G)$ be a cut vertex. Suppose that for any connected planar graph $G'$ of maximum degree $d$, which has fewer than $|V(G)|$ vertices, and for any vertex $t'\in V(G')$ whose degree is strictly smaller than $d$, there is a good drawing of $G'$ that is contained in the $t'$-wedge. Then $G$ also admits a good drawing that is contained in the $t$-wedge.
\end{lemma}
\begin{proof} Let $G_1, G_2,\ldots$ denote the connected components of the graph obtained from $G$ after the removal of the cut vertex $v$, and let $G^*_i$ be the subgraph of $G$ induced by $V(G_i)\cup\{v\}$.

If $t=v$ is a cut vertex, then by the induction hypothesis each $G_i^*$ has a good drawing in the $v$-wedge\footnote{Of course the $v$-wedges for the different components are different.}. After performing a suitable rotation for each of these drawings, and identifying their vertices corresponding to $v$, the lemma follows because the slopes of the segments incident to $v$ form a contiguous interval in each component.

If $t\ne v$, then let $G_j$ be the component containing $t$. Using the induction hypothesis, $G_j^*$ has a good drawing. Also, each $G_i^*$ for $i\ge 2$ has a good drawing in the $v$-wedge. As in the previous case, the lemma follows by rotating and possibly scaling down the components for $i\ne j$ and again identifying the vertices corresponding to $v$.
\end{proof}

In view of Lemma~\ref{cutclaim}, in the sequel we consider only biconnected graphs. We need the following definition.

\begin{defi} An ordering of the vertices of a graph, $v_1, v_2,\ldots,v_n$, is said to be an {\em $st$-ordering} if $v_1=s$, $v_n=t$, and if for every $1<i<n$ the vertex $v_i$ has at least one neighbor that precedes it and a neighbor that follows it.
\end{defi}

In \cite{LEC67}, it was shown that any biconnected graph has an $st$-ordering, for any choice of the vertices $s$ and $t$. In \cite{BK98}, this result was slightly strengthened for planar graphs, as follows.

\begin{lemma}\label{stlemma} {\rm (Biedl-Kant)} Let $D_G$ be a drawing of a biconnected planar graph, $G$, with vertices $s$ and $t$ on the outer face. Then $G$ has an $st$-ordering for which $s=v_1$, $t=v_n$ and $v_2$ is also a vertex of the outer face and $v_1v_2$ is an edge of the outer face.
\end{lemma}

We define $G_i$ to be the subgraph of $G$ induced by the vertices $v_1, v_2, \ldots, v_i$. Note that $G_i$ is connected. If $i$ is fixed, we call the edges between $V(G_i)$ and $V(G)\setminus V(G_i)$ the {\em pending edges}.
For a drawing of $G$, $D_G$, we denote by $D_{G_i}$ the drawing restricted to $G_i$ and to an initial part of each pending edge connected to $G_i$.

\begin{prop}\label{stremark0}
In the drawing $D_G$ guaranteed by Lemma \ref{stlemma}, $v_{i+1},\ldots v_n$ and the pending edges are in the outer face of $D_{G_i}$.
\end{prop}
\begin{proof}
Suppose for contradiction that for some $i$ and $j>i$, $v_j$ is not in the outer face of $D_{G_i}$. We know that $v_n$ is in the outer face of $D_{G_i}$ as it is on the outer face of $D_G$, thus $v_n$ and $v_j$ are in different faces of $D_{G_i}$. On the other hand, by the definition of $st$-ordering, there is a path in $G$ between $v_j$ and $v_n$ using only vertices from $V(G)\setminus V(G_i)$. The drawing of this path in $D_G$ must lie completely in one face of $D_{G_i}$. Thus, $v_j$ and $v_n$ must also lie in the same face, a contradiction.
Since the pending edges connect $V(G_i)$ and $V(G)\setminus V(G_i)$, they must also lie in the outer face.
\end{proof}

\begin{figure}[ht]
\centering
  \subfigure[The pending-order of the pending edges in $D_{G_i}$]
{		\includegraphics[scale=0.7]{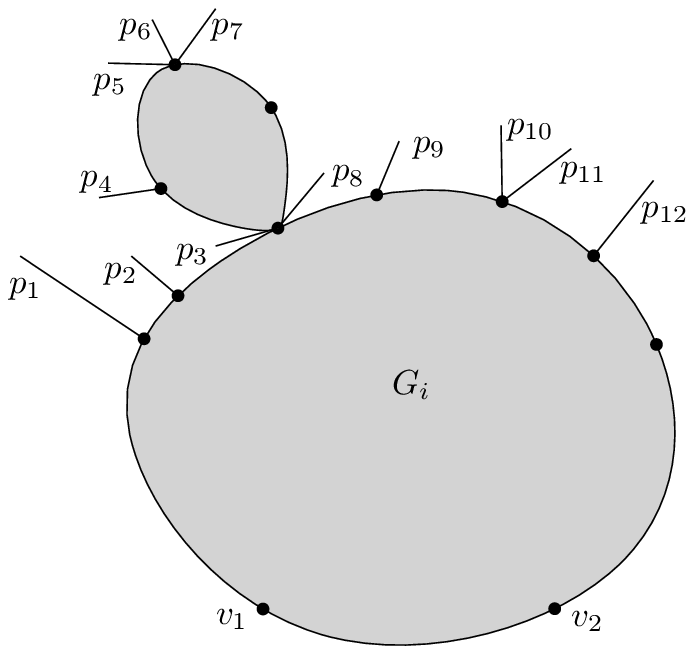}
\hspace{5mm}
				\label{fig:storder_a}
	}
       \hspace{2cm}
  \subfigure[The preceding neighbors of $v_{i+1}$ are consecutive in the pending-order]
{		
\includegraphics[scale=0.7]{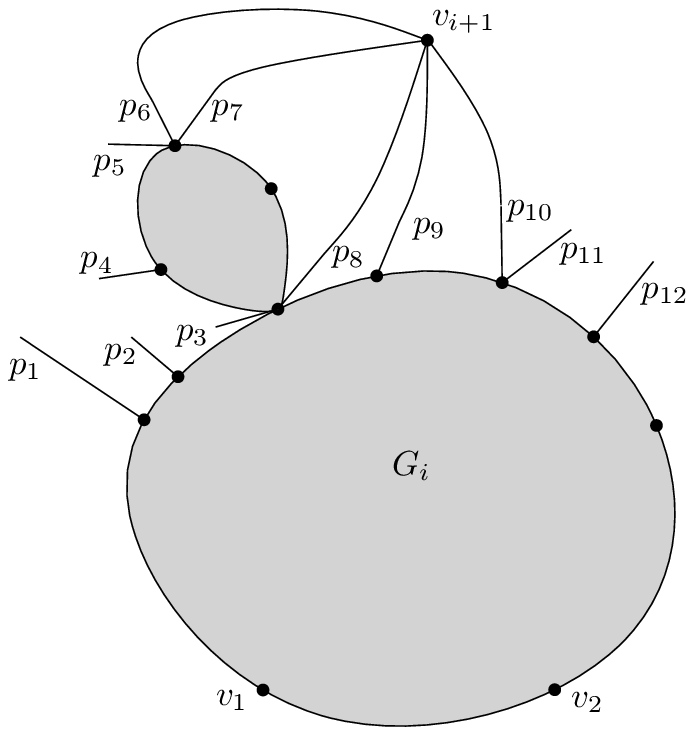}
\hspace{5mm}
				\label{fig:storder_b}
	}		 
		 \caption{Properties of the $st$-ordering}
\end{figure}

By Lemma \ref{stlemma}, the edge $v_1v_2$ lies on the boundary of the outer face of $D_{G_i}$, for any $i\ge 2$.
Thus, we can order the pending edges connecting $V(G_i)$ and $V(G)\setminus V(G_i)$ by walking in $D_G$ from $v_1$ to $v_2$ around $D_{G_i}$ on the side that does not consist of only the $v_1v_2$ edge, see Figure \ref{fig:storder_a}. We call this the {\em pending-order} of the pending edges between $V(G_i)$ and $V(G)\setminus V(G_i)$ (this order may depend on $D_G$). Proposition \ref{stremark0} implies 

\begin{prop}\label{stremark}
The edges connecting $v_{i+1}$ to vertices preceding it form an interval of consecutive elements in the pending-order of the edges between $V(G_i)$ and $V(G)\setminus V(G_i)$.
\end{prop}

For an illustration see Figure \ref{fig:storder_a}.

Two drawings of the same graph are said to be {\em equivalent} if the circular order of the edges incident to each vertex is the same in both drawings. Note that in this order we also include the pending edges (which are differentiated with respect to their yet not drawn end).

Now we are ready to finish the proof of Theorem \ref{twobends}, as the following lemma is the only missing step.

\begin{figure}[ht]
\centering
  \subfigure[Drawing $v_1$, $v_2$ and the edges incident to them]
  {
		\includegraphics[scale=0.36]{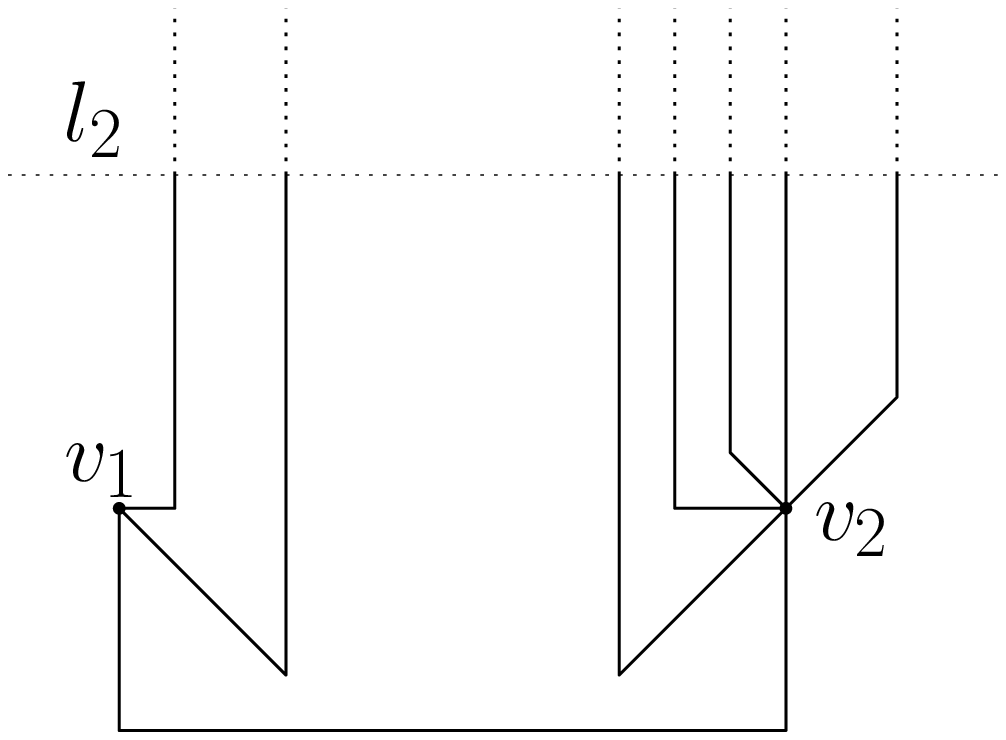}
               \hspace{5mm}
		 \label{fig:twobends}
      }
       \hspace{1cm}
      \subfigure[Adding $v_{i}$; partial edges added in this step are drawn with dashed lines]{
		\includegraphics[scale=0.36]{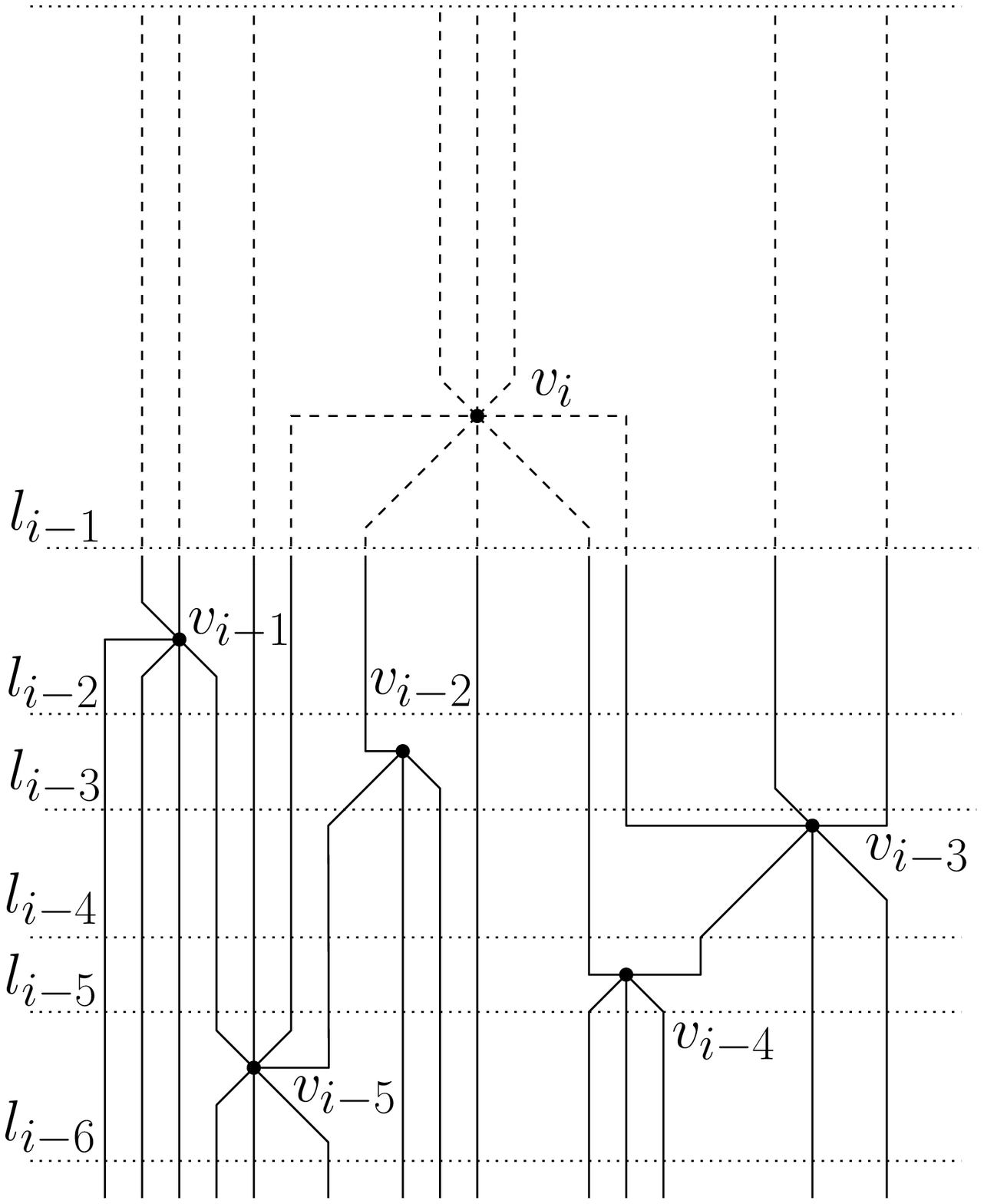}
               \hspace{5mm}
		 \label{fig:twobendsstep}
      }
      \caption{Drawing with at most two bends}
      \label{fig:twobendsboth}
\end{figure}

\begin{lemma}
For any biconnected planar graph $G$ with maximum degree $d\ge 6$ and for any vertex $t\in V(G)$ with degree strictly less then $d$, $G$ admits a good drawing that is contained in the $t$-wedge.
\end{lemma}
\begin{proof}
Take a planar drawing $D_G$ of $G$ such that $t$ is on the outer face and pick another vertex, $s$, from the outer face. Apply Lemma \ref{stlemma} to obtain an $st$-ordering with $v_1=s, v_2,$ and $v_n=t$ on the outer face of $D_G$ such that $v_1v_2$ is an edge of the outer face. We will build up a good drawing of $G$ by starting with $v_1$ and then adding $v_2,v_3,\ldots, v_n$ one by one to the outer face of the current drawing. As soon as we add a new vertex $v_i$, we also draw the initial pieces of the pending edges, and we make sure that the resulting drawing is equivalent to the drawing $D_{G_i}$.

Another property of the good drawing that we maintain is that every edge consists of precisely three pieces. (Actually, an edge may consist of fewer than 3 segments, because two consecutive pieces are allowed to have the same slope and form a longer segment) The middle piece will always be vertical, except for the middle piece of $v_1v_2$.

Suppose without loss of generality that $v_1$ follows directly after $v_2$ in the clockwise order of the vertices around the outer face of $D_G$. Place $v_1$ and $v_2$ arbitrarily in the plane so that the $x$--coordinate of $v_1$ is smaller than the $x$--coordinate of $v_2$. Connect $v_1$ and $v_2$ by an edge consisting of three segments: the segments incident to $v_1$ and $v_2$ are vertical and lie below them, while the middle segment has an arbitrary non-vertical regular slope.
Draw a horizontal auxiliary line $l_2$ above $v_1$ and $v_2$. Next, draw the initial pieces of the other (pending) edges incident to $v_1$ and $v_2$, as follows. For $i=1,2$, draw a short segment from $v_i$ for each of the edges incident to it (except for the edge $v_1v_2$, which has already been drawn) so that the directed slopes of the edges (including $v_1v_2$) form a contiguous interval and their circular order is the same as in $D_G$. Each of these short segments will be followed by a vertical segment that reaches above $l_2$. These vertical segments will belong to the middle pieces of the corresponding pending edges. Clearly, for a proper choice of the lengths of the short segments, no crossings will be created during this procedure. So far this drawing, including the partially drawn pending edges between $V(G_2)$ and $V(G)\setminus V(G_2)$, will be equivalent to the drawing $D_{G_2}$. As the algorithm progresses, the vertical segments will be further extended above $l_2$, to form the middle segments of the corresponding edges. For an illustration, see Figure \ref{fig:twobends}.

The remaining vertices $v_{i}, i>2$, will be added to the drawing one by one, while maintaining the property that the drawing is equivalent to $D_{G_i}$ and that the pending-order of the actual pending edges 
coincides with the order in which their vertical pieces reach the auxiliary line $l_i$. At the beginning of step $i+1$, these conditions are obviously satisfied. Now we show how to place $v_{i+1}$. 

Consider the set $X$ of intersection points of the vertical (middle) pieces of all pending edges between 
$V(G_i)$ and $V(G)\setminus V(G_i)$ with the auxiliary line $l_i$. By Proposition \ref{stremark}, the intersection points corresponding to the pending edges incident to $v_{i+1}$ must be consecutive elements of $X$. Let $m$ be (one of) the median element(s) of $X$. Place $v_{i+1}$ at a point above $m$, so that the $x$-coordinates of $v_{i+1}$ and $m$ coincide, and connect it to $m$. (In this way, the corresponding edge has only one bend, because its second and third piece are both vertical.) We also connect $v_{i+1}$ to the upper endpoints of the appropriately extended vertical segments passing through the remaining elements of $X$, so that the directed slopes of the segments leaving $v_{i+1}$ form a contiguous interval of regular slopes. For an illustration see Figure \ref{fig:twobendsstep}.
Observe that this step can always be performed, because, by the definition of $st$-orderings, the number of edges leaving $v_{i+1}$ is strictly smaller than $d$. This is not necessarily true in the last step, but then we have $v_n=t$, and we assumed that the degree of $t$ was smaller than $d$.
To complete this step, draw a horizontal auxiliary line $l_{i+1}$ above $v_{i+1}$ and extend the vertical portions of those pending edges between $V(G_i)$ and $V(G)\setminus V(G_i)$ that were not incident to $v_{i+1}$ until they hit the line $l_{i+1}$. (These edges remain pending in the next step.)
Finally, in a small vicinity of $v_{i+1}$, draw as many short segments from $v_{i+1}$ using the remaining directed slopes as many pending edges connect $v_{i+1}$ to $V(G)\setminus V(G_{i+1})$. Make sure that the directed slopes used at $v_{i+1}$ form a contiguous interval and the circular order is the same as in $D_G$. Continue each of these short segments by adding a vertical piece that hits the line $l_{i+1}$. The resulting drawing, including the partially drawn pending edges, is equivalent to $D_{G_{i+1}}$.

In the final step, if we place the auxiliary line $l_{n-1}$ high enough, then the whole drawing will be contained in the $v_n$-wedge and we obtain a drawing that meets the requirements.
\end{proof}

\section{Lower Bounds}
In this section, we construct a sequence of planar graphs, providing a nontrivial lower bound for the planar slope number of bounded degree planar graphs. They also require more than the trivial number ($\lceil d/2\rceil$) slopes, even if we allow one bend per edge. 
Remember that if we allow {\em two} bends per edge, then, by Theorem \ref{twobends}, for all graphs with maximum degree $d\ge 3$, except for the octahedral graph, $\lceil d/2\rceil$ slopes are sufficient, which bound is optimal.

\medskip

\begin{thm} For any $d\ge 3$, there exists a planar graph $G_d$ with maximum degree $d$, whose planar slope number is at least $3d-6$.
In addition, any drawing of $G_d$ with at most one bend per edge requires at least $\frac 34(d-1)$ slopes.
\end{thm}

\medskip

\begin{figure}[ht]
\centering
  \subfigure[A straight line drawing of $G_6$]{
		\includegraphics[scale=0.5]{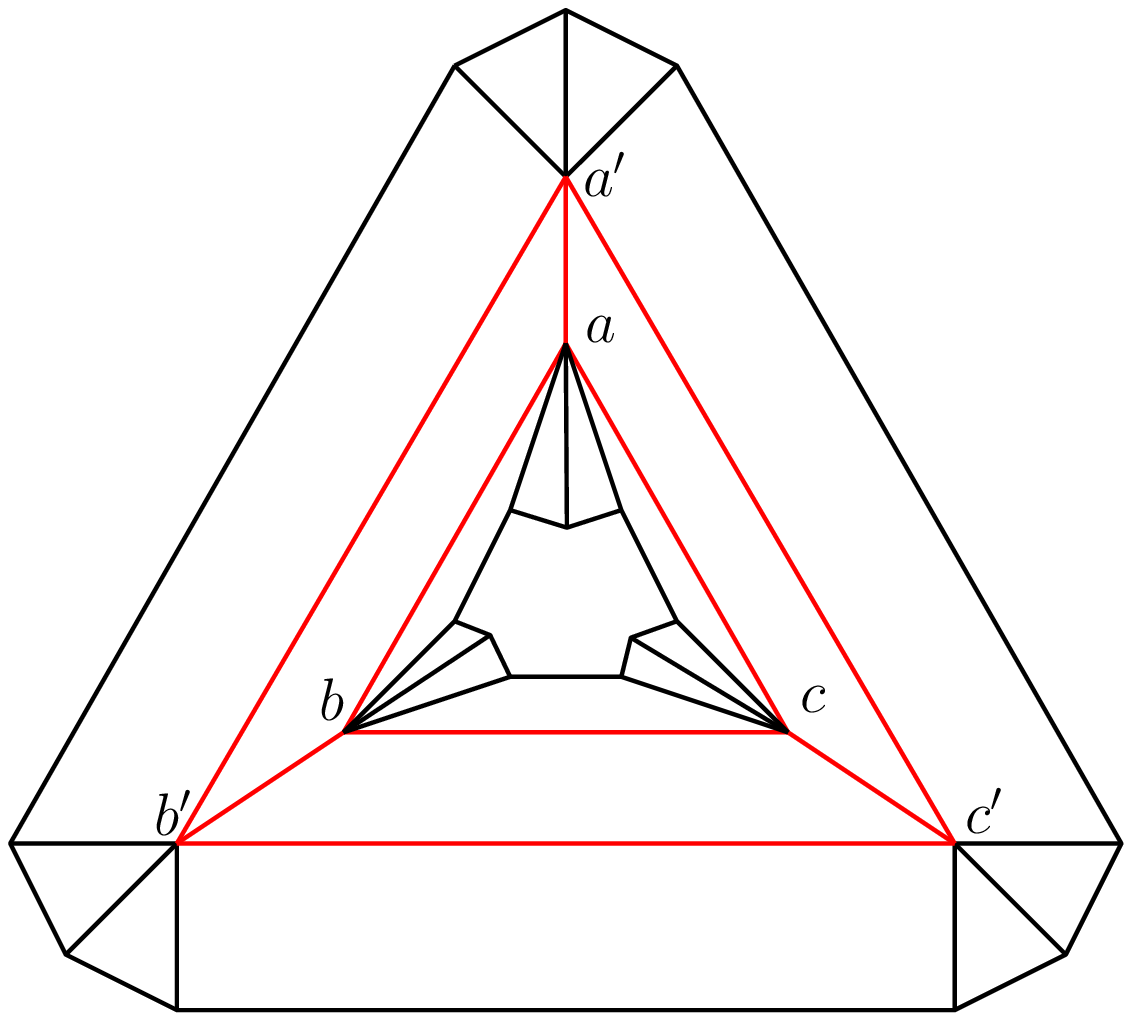}
       \label{fig:counterex}
      }		
	\hspace{2cm}
  \subfigure[At most four segments starting from $a,b,c$ can use the same slope in a drawing of $G_d$ with one bend per edge]
  {
		\includegraphics[scale=0.4]{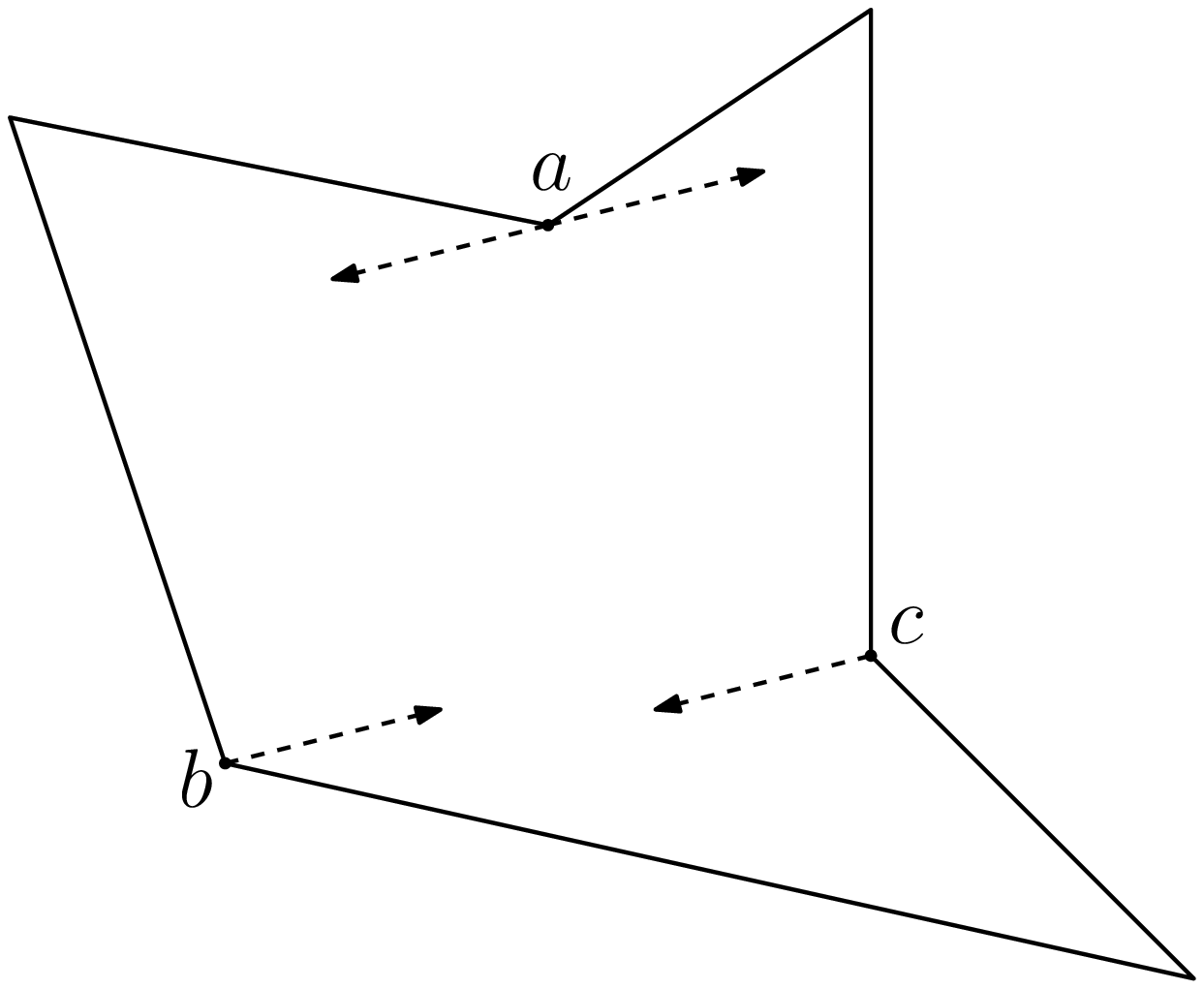}
      \label{fig:counterex1bend}
      }
		  \caption{Lower bounds}
\end{figure}

\begin{proof}
The construction of the graph $G_d$ is as follows. Start with a graph of $6$ vertices, consisting of two triangles, $abc$ and $a'b'c'$, connected by the edges $aa'$, $bb'$, and $cc'$ (see Figure \ref{fig:counterex}). Add to this graph a cycle $C$ of length $3(d-3)$, and connect $d-3$ consecutive vertices of $C$ to $a$, the next $d-3$ of them to $b$, and the remaining $d-3$ to $c$. Analogously, add a cycle $C'$ of length $3(d-3)$, and connect one third of its vertices to $a'$, one third to $b'$, one third to $c'$. In the resulting graph, $G_d$, the maximum degree of the vertices is $d$.

In any crossing-free drawing of $G_d$, either $C$ lies inside the triangle $abc$ or $C'$ lies inside the triangle $a'b'c'$. Assume by symmetry that $C$ lies inside $abc$, as in Figure \ref{fig:counterex}.

If the edges are represented by straight-line segments, the slopes of the edges incident to $a,b,$ and $c$ are all different, except that $aa', bb',$ and $cc'$ may have the same slope as some other edge. Thus, the number of different slopes used by any straight-line drawing of $G_d$ is at least $3d-6$.

Suppose now that the edges of $G_d$ are represented by polygonal paths with at most one bend per edge. Assume, for simplicity, that every edge of the triangle $abc$ is represented by a path with exactly one bend (otherwise, an analogous argument gives an even better result). Consider the $3(d-3)$ polygonal paths connecting $a$, $b$, and $c$ to the vertices of the cycle $C$. Each of these paths has a segment incident to $a$, $b$, or $c$. Let $S$ denote the set of these segments, together with the $6$ segments of the paths representing the edges of the triangle $abc$.

\begin{claim} \label{4slopelemma}
The number of segments in $S$ with any given slope is at most $4$.
\end{claim}
\begin{proof}
The sum of the degrees of any polygon on $k$ vertices is $(k-2)\pi$.
Every direction is covered by exactly $k-2$ angles of a $k$-gon (counting each side $1/2$ times at its endpoints). Thus, if we take every other angle of a hexagon, then, even including its sides, every direction is covered at most $4$ times. (See Figure \ref{fig:counterex1bend}.)
\end{proof}

The claim now implies that for any drawing of $G$ with at most one bend per edge, we need at least $(3(d-3)+6)/4=\frac {3}{4}(d-1)$ different slopes.
\end{proof}

\section*{Acknowledgements}
We are grateful to G\"unter Rote for pointing out a mistake in the original proof of Lemma~\ref{cutclaim} and to an anonymous referee for calling our attention to \cite{M93}.


\begin{thebibliography}{99}\fontsize{10}{0}


\bibitem{BMW06} J.~Bar\'at, J.~Matou\v sek, and D.~Wood: Bounded-degree
graphs have arbitrarily large geometric thickness, {\it Electronic
J. Combinatorics} {\bf 13}/1 (2006), R3.

\bibitem{BK98} T.~Biedl and G.~Kant: A better heuristic for orthogonal graph drawings, {\it Comput. Geom.} 9 (1998), pp. 159-180.



\bibitem{DETT99} G.~Di~Battista, P.~Eades, R.~Tamassia, and I.~G.~Tollis:
{\it Graph Drawing}, Prentice Hall, Upper Saddle River, N.J., 1999.

\bibitem{DESW07} V.~Dujmovi\'c, D.~Eppstein, M.~Suderman, and D.~R.~Wood:
Drawings of planar graphs with few slopes and segments,
{\em Comput. Geom.} {\bf 38} (2007), no. 3, 194--212.

\bibitem{DSW07} V.~Dujmovi\'c, M.~Suderman, and D.~R.~Wood:
Graph drawings with few slopes, {\em Comput. Geom.} {\bf 38} (2007), 181--193.

\bibitem{DEK04} C.~A.~Duncan, D.~Eppstein, and S.~G.~Kobourov: The
geometric thickness of low degree graphs. In: {\it Proc. 20th ACM Symp.
on Computational Geometry (SoCG'04)}, ACM Press, 2004, 340--346.

\bibitem{En05} M.~Engelstein: Drawing graphs with few slopes,
Research paper submitted to the Intel Competition for high school
students, New York, October 2005.

\bibitem{E04} D.~Eppstein: Separating thickness from geometric
thickness, {\it Towards a Theory of Geometric Graphs (J.~Pach,
ed.), Contemporary Mathematics} {\bf 342}, Amer. Math. Soc,
Providence, 2004, 75--86.

\bibitem{F48} I. F\'ary, On straight line representation of planar graphs, {\em Acta Univ. Szeged. Sect. Sci. Math.} {\bf 11} (1948), 229--233.

\bibitem{FMR94} H.~de~Fraysseix, P.~Ossona~de~Mendez, P.~Rosenstiehl: On triangle contact graphs, {\em Combinatorics, Probability and Computing} {\bf 3} (1994), 233--246.

\bibitem{FPP89} H.~de~Fraysseix, J.~Pach, and R.~Pollack:
How to draw a planar graph on a grid, {\em Combinatorica} {\bf 10}(1) (1990), 41--51.

\bibitem{GT94} A.~Garg and R.~Tamassia:
Planar drawings and angular resolution: Algorithms and bounds, {\em Second Annual European Symposium, Lecture Notes in Computer Science} {\bf 855}, Springer, 1994, 12--23.

\bibitem{JJ10} V. Jelinek, E. Jelinkov\'a, J. Kratochv\'{\i}l, B. Lidick\'y, M. Tesa\v r, and T. Vysko\v cil: The planar slope number of planar partial 3-trees of bounded degree. In: {\em Graph Drawing, Lecture Notes in Computer Science} {\bf 5849}, Springer, Berlin, 2010, 304--315.

\bibitem{K36} P.~Koebe: Kontaktprobleme der konformen Abbildung,  {\em Berichte Verhand. Sächs. Akad. Wiss. Leipzig, Math.-Phys. Klasse} {\bf  88} (1936) 141--164.

\bibitem{KPPT08} B.~Keszegh, J.~Pach, D.~Pálvölgyi, and G.~Tóth: Drawing cubic graphs with at most five slopes, Drawing cubic graphs with at most five slopes. {\em Comput. Geom.} {\bf 40} (2008), no. 2, 138--147.

\bibitem{LEC67} A.~Lempel, S.~Even, I.~Cederbaum: An algorithm for planarity testing of graphs, {\em Theory of Graphs} (P.~Rosenstiehl, ed.), Gordon and Breach, New York, 1967, 215--232 .

\bibitem{LMS91} Y.~Liu, A.~Morgana, B.~Simeone: General theoretical results on rectilinear embeddability of graphs, {\em Acta Math. Appl. Sinica} {\bf 7} (1991), 187--192.


\bibitem{MP94} S.~Malitz and A.~Papakostas: On the angular resolution of planar graphs,  {\em SIAM J. Discrete Math.} {\bf  7} (1994), 172-183.

\bibitem{M93} B.~Mohar, A polynomial time circle packing algorithm, Discrete Math., 117 (1993), 257-263.

\bibitem{MSz07} P.~Mukkamala and M.~Szegedy: Geometric
    representation of cubic graphs with four directions,
    {\it Comput. Geom.} {\bf 42} (2009), 842--851.

\bibitem{PP06} J.~Pach and D.~P\'alv\"olgyi: Bounded-degree graphs can
have arbitrarily large slope numbers, {\it Electronic J.
Combinatorics} {\bf 13}/1 (2006), N1.

\bibitem{PT06} J.~Pach and G.~T\'oth: Crossing number of toroidal graphs,  {\em Graph drawing,   Lecture Notes in Comput. Sci.} {\bf 3843}, Springer, Berlin, 2006, 334--342.

\bibitem{U53} P.~Ungar: On diagrams representing maps, {\it J. London
Math. Soc.} {\bf 28} (1953), 336--342.

\bibitem{WC94} G. A.~Wade and J. H.~Chu: Drawability of complete graphs
using a minimal slope set, {\it The Computer J.} {\bf 37} (1994),
139--142.

\end{thebibliography}
\end{document}